\documentclass[10pt]{article}
\usepackage{graphicx}
\usepackage{amsmath}
\usepackage{amsfonts}
\usepackage{amssymb}
\usepackage{amsthm} 
\usepackage{enumerate}
\usepackage[cp1250]{inputenc}
\usepackage[cm]{fullpage}
\usepackage[a4paper, left=3.0cm, right=3.0cm, top=3.0cm, bottom=3.0cm, headsep=1.2cm]{geometry}

\usepackage{caption}
\usepackage{color}
\definecolor{sepia}{cmyk}{0, 0.83, 1, 0.70}

\usepackage[plainpages=false,pdfpagelabels,bookmarksnumbered,%
colorlinks=true,%
linkcolor=sepia,%
citecolor=sepia,%
unicode]{hyperref}

\newtheorem{theorem}{Theorem}

\newtheorem{corollary}[theorem]{Corollary}

\newtheorem{lemma}[theorem]{Lemma}

\theoremstyle{definition}
\newtheorem{definition}[theorem]{Definition}
\newtheorem{remark}[theorem]{Remark}
\newtheorem{example}[theorem]{Example}

\numberwithin{equation}{section} 
\numberwithin{theorem}{section}  
\numberwithin{figure}{section}   
\numberwithin{table}{section}    

\makeatletter
\renewenvironment{proof}[1][\proofname]
{\par
	\pushQED{$\blacksquare$} 
	\normalfont\topsep6\p@\@plus6\p@\relax
	\trivlist
	\item[\hskip\labelsep\bfseries#1\@addpunct{.}]
	\ignorespaces}
{\popQED \endtrivlist\@endpefalse}
\makeatother

\DeclareMathOperator*{\interior}{int}

\DeclareMathOperator*{\fix}{Fix}

\DeclareMathOperator*{\argmin}{argmin}

\begin{document}
\title{\Large\vspace{-4em}\textbf{Finitely Convergent Iterative Methods with Overrelaxations \\Revisited}}
\author{Victor I. Kolobov\thanks{Department of Computer Science, The Technion -- Israel Institute of Technology, 32000 Haifa, Israel, \texttt{kolobov.victor@gmail.com}}
\and Simeon Reich\thanks{Department of Mathematics, The Technion -- Israel Institute of Technology, 32000 Haifa, Israel, \texttt{sreich@technion.ac.il}}
\and Rafa\l\ Zalas\thanks{Department of Mathematics, The Technion -- Israel Institute of Technology, 32000 Haifa, Israel, \texttt{zalasrafal@gmail.com}}}

\maketitle
	
\begin{abstract}

We study the finite convergence of iterative methods for solving convex feasibility problems. Our key assumptions are that the interior of the solution set is nonempty and that certain overrelaxation parameters converge to zero, but with a rate slower than any geometric sequence. Unlike other works in this area, which require divergent series of overrelaxations, our approach allows us to consider some summable series. By employing quasi-Fej\'{e}rian analysis in the latter case, we obtain additional asymptotic convergence guarantees, even when the interior of the solution set is empty.
		
\end{abstract}
	
\section{Introduction}
We consider the following consistent convex feasibility problem (CFP) in a real Hilbert space $\mathcal{H}$:
\begin{equation}\label{cfp}
	\text{Find}\quad x\in C\cap Q\quad\text{with}\quad C:=\bigcap_{i\in I}C_i,
\end{equation}
where each one of the sets $C_i \subset \mathcal H$, $i\in I:=\{1,\ldots,m\}$, $m \in \mathbb N_+$, as well as the set $Q \subset \mathcal H$, are closed and convex.  We represent each $C_i$ as the fixed point set  of some operator $T_i\colon\mathcal{H}\rightarrow\mathcal{H}$, that is, $C_i = \fix T_i:=\{x \colon T_i(x)=x\}$. We restrict our considerations to a very broad class of operators, which, following \cite{Ceg12}, we call here \emph{cutters}; see Definition \ref{def:QNE} below. Since these operators can be considered a variation of firmly nonexpansive operators \cite[Theorem 2.2.5]{Ceg12}, they are called \emph{firmly quasi-nonexpansive} by some authors; see \cite{BC17}. Such an approach allows us to use  not only metric projections $P_{C_i}$, which indeed are  firmly  nonexpansive, but also subgradient projections $P_{f_i}$, which may happen to be discontinuous \cite[Example 29.47]{BC17}. The latter operators are particularly useful when applied to the inequality constraints $C_i = \{x \colon f_i(x)\leq 0\}$ defined by weakly lower semicontinuous functions $f_i \colon \mathcal{H} \to \mathbb{R}$, $i\in I$; see Example \ref{ex:subProj}.
	
We study a class of iterative methods with the sequence of approximations $\{x_k\}_{k=0}^\infty$ defined by
\begin{equation}\label{int:xk}
  x_0\in Q,\quad x_{k+1}:=
  P_Q\left(x_k+\alpha_{k}\sum_{i\in I_k}
  \lambda_{i,k}
  \beta_{i,k}(x_k)
  \Big(T_i(x_k)-x_k\Big)\right),
\end{equation}
where $\alpha_k \in (0, 2)$ are \emph{relaxation parameters}, the weights $\lambda_{i,k} \in (0, 1]$ satisfy $\sum_{i\in I_k} \lambda_{i,k} = 1$ and $\{I_k\}_{k=0}^\infty$ is a given \emph{control sequence} in $I$, that is, $I_k\subset I$. Prototypical versions of method \eqref{int:xk} with functionals  $\beta_{i,k}(x) = \beta_k(x) \geq 1$  and $Q = \mathcal H$ can be found in many papers; see, for example, \cite{BB96, BNPH15, BLT15, Ceg12, CRZ18, Com96, Com01, CET20, KRZ17, ZNL18}. Note here that in all of these works, the authors discuss asymptotic convergence such as weak or norm convergence, with some of them providing polynomial or linear error bounds. However, there are situations where the produced iterates may never be feasible; see, for example, \cite[Theorem 7]{LTT20}.

In this paper we consider a variation of \eqref{int:xk}, the main idea of which is the extension of each nonzero vector $T_i(x_k) - x_k$ by using the scalar $\beta_{i,k}(x_k) >1$. Following \cite{KRZ20}, we define $\beta_{i,k}\colon \mathcal H \to [0,\infty)$ by
\begin{equation}\label{int:betak}
  \beta_{i,k}(x):=
  \begin{cases}\displaystyle
    \frac{\frac{r_k}{ \varphi_i(x)} +\|T_i (x)-x\|}{\|T_i(x)-x\|},
    & \text{if } T_i(x)\neq x \\
    0, & \text{otherwise,}
  \end{cases}
\end{equation}
where $r_k \in (0,\infty)$ are the \emph{overrelaxation parameters} and $\varphi_i \colon \mathcal H \to (0,\infty)$, $i \in I$, are the \emph{overrelaxation functionals}.

The above definition captures two important instances of $\beta_{i,k}$, which have been studied in the literature. In the first one, we have $\varphi_i(x) := 1$; see, for example, \cite{BWWX15, Cro04, Pol01}. The second instance corresponds to inequality constraints, where  $\varphi_i(x) := \|g_i(x)\|$ for all $x \notin C_i$, $g_i(x) \in \partial f_i(x)$ and where $T_i:= P_{f_i}$; see, for example, \cite{CCP11, DePI88, IM86}. Clearly, the latter case motivates the use of overrelaxation functionals. For a more detailed description of these methods, see \cite[Table 1.1]{KRZ20}.

It was shown in \cite[Theorem 4.7 and Remark 4.11]{KRZ20} that, under some conditions, all bounded trajectories must reach the solution set within a finite number of steps, a property to which we refer as \emph{finite convergence}. The two main conditions on which we would like to focus in this brief introduction are: (i) the constraint qualification $\interior(C) \cap Q \neq \emptyset$, and (ii)
\begin{equation}\label{int:rk}
	r_k\downarrow 0\qquad\text{and}\qquad \sum_{k=0}^\infty r_k=+\infty.
\end{equation}
The down-arrow symbol ``$\downarrow$'' stands here for monotone convergence, where $r_{k+1} \leq r_k$.

In the present paper we investigate the finite convergence properties of method \eqref{int:xk}--\eqref{int:betak} assuming, as in (ii), that indeed $r_k \downarrow 0$ but with a \emph{rate slower than any geometric sequence}, that is:
\begin{equation}\label{int:STAG}
  \text{\emph{For each $q\in(0,1)$ and $c>0$ there is $K \geq 0$ s.t. $r_k > cq^k$ for all $k \geq K$.}}
\end{equation}
This holds, in particular, when the rate is \emph{sublinear}, that is, when $r_{k+1}/r_k \to 1$ as $k \to \infty$. In Theorem \ref{thm:main}, which is the main result of our paper, we show that this version of condition (ii) also leads to finite convergence. At this point we note that our analysis is closely aligned with arguments used in \cite{CCP11, DePI88, IM86}, where condition \eqref{int:STAG} did appear. On the other hand, our analysis differs from \cite{BWWX15, Cro04, KRZ20, Pol01}, which relied heavily on \eqref{int:rk}.

The main adavantage of \eqref{int:STAG} over \eqref{int:rk} is that some sequences of overrelaxations satisfying \eqref{int:STAG} may still form a summable series $\sum_{k=0}^{\infty} r_k < \infty$. Take, for example, $r_k:=\frac{1}{k^\alpha}$ with $\alpha>1$. This turns out to be an important factor in the study of the asymptotic convergence of method \eqref{int:xk}--\eqref{int:betak}, when considered under simplified conditions. In particular, this includes constraint qualification reduced to just $C \cap Q \neq \emptyset$. Indeed, in this case, we have shown that all the bounded trajectories $\{x_k\}_{k=0}^\infty$ are quasi-Fej\'er monotone sequences of type I with respect to $C \cap Q$. Consequently, by applying the apparatus of quasi-Fej\'erian analysis developed in \cite{Com01}, we have formulated sufficient conditions for weak and norm convergence of \eqref{int:xk}--\eqref{int:betak}. Note here that in both cases, the assumed regularities of the operators are weaker than the bounded linear regularity of Theorem \ref{thm:main}. This constitutes our second main result which is presented in Theorem \ref{thm:main2}.

We would like to emphasize here that both our results, Theorem \ref{thm:main} and Theorem \ref{thm:main2}, share a common instance of sufficient conditions. Thus, both of these theorems, when combined, certify asymptotic and, moreover, finite convergence of method \eqref{int:xk}--\eqref{int:betak} depending on, a possibly unknown, form of the constraint qualification. This was not the case in \cite{BWWX15, CCP11, Cro04, DePI88, IM86, IM87, KRZ20}, where $\interior (C) \cap Q \neq \emptyset$, or even the stronger Slater condition, were essential. We provide an example of such a result in Corollary \ref{cor:proj}.

Despite of the above-mentioned advantage, there are a few aspects, where our results are less general than those obtained in \cite{KRZ20}. This applies, in particular, to the use of intermittent controls, which excludes repetitive and random control sequences, both of which were allowed in \cite{KRZ20}. Furthermore, Theorem \ref{thm:main} requires bounded linear regularity of the operators $T_i$, which was not necessary in \cite{KRZ20}. We also note that bounded linear regularity was not explicitly mentioned in \cite{CCP11, DePI88, IM86, IM87}. However, it was satisfied therein due to Slater's condition and \cite[Example 2.11]{CRZ20}. On the other hand, the regularity of operators and sets has now become a standard assumption used in the analysis of the basic Fej\'er monotone methods; see, for example, \cite{BB96, BNPH15, BLT15, CRZ18, CET20, KRZ17}.

Despite its general form, our framework \eqref{int:xk}--\eqref{int:betak} does not capture many instances of iterative methods  which have the finite convergence property.  See, for example, \cite{Fuk82, Pan14}, where the main step of the iterative method consists of a projection onto a polyhedral approximation of the inequality constraints obtained by using sublinear overrelaxations. See also \cite{IM87} which employs an abstract, Fej\'{e}r monotone, algorithmic operator and divergent series of overrelaxations. Other examples of iterative methods, which do not involve any overrelaxations, can be found, for example, in \cite{BD17, BDNP16} which discusses the Douglas-Rachford method, in \cite{Pan15} which presents a variation of the Haugazeau method and in \cite{BBHK20} which concerns Dykstra's algorithm.

The organization of our paper is as follows. In Section \ref{sec:preliminaries} we present a few technical results, which facilitate our study. In Section \ref{sec:finiteConv} we present our main result regarding finite convergence. In Section \ref{sec:asymptotic} we discuss the asymptotic behaviour of our method. In the last section, we provide an example where both of our main results overlap.

\section{Preliminaries} \label{sec:preliminaries}
	
Let $C$ be a nonempty proper subset of $\mathcal{H}$. The \emph{distance functional} $d(\cdot,C)\colon\mathcal{H}\to[0,\infty)$ is defined by $d(x,C):=\inf\{\|x-z\|:z\in C\}$, $x\in\mathcal{H}$. The \emph{signed distance functional} $\varphi(\cdot,C) \colon \mathcal{H}\to(-\infty,\infty)$ is defined by $\varphi(x,C):=d(x,C)-d(x,\mathcal{H}\setminus C)$, $x\in\mathcal{H}$. It is well known that when $C$ is closed and convex, then both functionals, $d(\cdot,C)$ and $\varphi(\cdot,C)$, are convex and $1$-Lipschitz; see \cite[Page 940]{BR80} for $d(\cdot,C)$ and \cite{HU77} for $\varphi(\cdot,C)$. See also the more recent \cite{LWL19}.
	
For a given function $f\colon\mathcal{H}\rightarrow\mathbb{R}$, the sublevel set at level zero is defined by
\begin{equation} \label{def:sublevelSet}
	S_f:=\{x\in\mathcal{H}:f(x)\leq 0\}.
\end{equation}
	
\begin{lemma}\label{lem:erosion}
  Let $f\colon\mathcal H \to \mathbb R$ be a convex function, let $0\leq\varepsilon\leq r$ and assume that $S_{f+r} \neq \emptyset$. Then for all $x \in \mathcal{H} \setminus S_f$, we have
	\begin{equation}\label{lem:eqerosion:ineq}
    \frac{d(x,S_{f+\varepsilon})}{f(x)+\varepsilon} \leq \frac{d(x,S_{f+r})}{f(x)+r}.
	\end{equation}
\end{lemma}

\begin{proof}
Let $z\in S_{f+r}$ and let
\begin{equation}
	\gamma:=1-\frac{f(x)+\varepsilon}{f(x)-f(z)}.
\end{equation}
One can verify that $0\leq\gamma\leq 1$. Let $y := (1-\gamma)z+\gamma x$. Using the convexity of $f(\cdot)$, we obtain
\begin{equation}
  f(y) \leq (1-\gamma)f(z) + \gamma f(x)
	= \frac{f(x)f(z) + \varepsilon f(z)-f(z)f(x)-\varepsilon f(x)}{f(x)-f(z)}
	= -\varepsilon.
\end{equation}
Hence $y\in S_{f+\varepsilon}$. Thus we get
\begin{equation}
	d(x,S_{f+\varepsilon})\leq\|x-y\|
	= (1-\gamma)\|x-z\|
	= \frac{f(x)+\varepsilon}{f(x)-f(z)}\|x-z\|
	\leq (f(x)+\varepsilon)\frac{\|x-z\|}{f(x)+r}.
\end{equation}
Now choose $z=P_{S_{f+r}}(x)$ to obtain
\begin{equation}
  d(x,S_{f+\varepsilon}) \leq (f(x)+\varepsilon) \frac{d(x,S_{f+r})}{f(x)+r}.
\end{equation}
\end{proof}

In particular, when $f = \varphi(\cdot,C)$ for a closed and convex set $C$, then the sublevel set $S_{f+\varepsilon}$ becomes the $\varepsilon$-\emph{erosion} of $C$ given by
\begin{equation}\label{def:erosion}
  C^\varepsilon := \{ x \in C \colon B(x,\varepsilon) \subset C\},
\end{equation}
where $B(x,\varepsilon):=\{z \in \mathcal H \colon \|z-x\|\leq \varepsilon\}$ is a closed ball and inequality \eqref{lem:eqerosion:ineq} can be equivalently written as
\begin{equation}\label{eq:erosion}
  \frac{d(x,C^\varepsilon)}{d(x,C)+\varepsilon} \leq \frac{d(x,C^r)}{d(x,C)+r}.
\end{equation}

\begin{lemma}\label{lem:subdiff}
For each $k\in\mathbb{N}$, let $f_k\colon\mathcal{H}\rightarrow\mathbb{R}$ be a convex and lower semicontinuous function, and assume that $f(z) := \sup_{k\in\mathbb{N}} f_k (z)<0$ for some $z\in\mathcal H$. Then for all $r>0$, we have
\begin{equation}\label{lem:subdiff:ineq}
  \inf \Big\{\|g_k(x)\| \colon  x\in B(z,r),\ f_k(x) \geq 0 \text{ and } g_k(x) \in \partial f_k(x)\Big\} \geq \frac{-f(z)}{r}=:\lambda>0.
\end{equation}
\end{lemma}
	
\begin{proof}
  See \cite[Lemma 3.3]{KRZ20}.
\end{proof}

\subsection{Quasi-nonexpansive operators}

\begin{definition}\label{def:QNE}
Let $T:\mathcal H\rightarrow\mathcal H$ be an operator with a fixed point,  that is, $\fix T=\{z\in \mathcal H\colon z=T(z)\} \neq \emptyset$. We say that $T$ is
\begin{enumerate}[(i)]
	\item \textit{quasi-nonexpansive} (QNE) if for all $x\in\mathcal H$ and
  all $z\in\fix T$, we have $\| T (x) -z\|\leq\| x-z\|$.
			
	\item $\rho$\textit{-strongly quasi-nonexpansive} ($\rho$-SQNE), where
  $\rho\geq 0$, if for all $x\in\mathcal H$ and all $z\in\fix T$, we have $\| T(x)-z\|^2\leq\| x-z\|^2-\rho\|T(x) -x\|^2$.
			
	\item a \textit{cutter} if for all $x\in\mathcal H$ and all $z\in\fix T$, we have
  $\langle z-T(x) ,x-T(x) \rangle\leq 0$.
	\end{enumerate}
\end{definition}

 Note that the operators mentioned in Definition \ref{def:QNE} can be found under different names in the literature. For example, cutters appear as $\mathcal T$-\emph{class} in \cite{BC01, Com01} and \emph{firmly quasi-nonexpansive} operators in \cite{BC17}. For more details concerning this topic, we refer the interested reader to \cite[pages 47 and 53--54]{Ceg12}.  A comprehensive review of the properties of QNE, SQNE and cutter operators can be found,  for example,  in \cite[Chapter 2]{Ceg12}.

\begin{example}[Metric Projection]
  Let $C\subset \mathcal H$ be nonempty, closed and convex. The \emph{metric projection} $P_C(x):=\argmin_{z\in C}\|z-x\|$  is a cutter and  $\fix P_C = C$; see, for example, \cite[Theorem 1.2.4]{Ceg12}. We remark in passing that $P_C$ is actually firmly nonexpansive \cite[page 18]{GR84}.
\end{example}

\begin{example}[Subgradient Projection]\label{ex:subProj}
  Let $f\colon \mathcal{H}\to \mathbb{R}$ be a lower semicontinuous and convex function with nonempty sublevel set $S(f,0):=\{x\in \mathcal{H}\colon f(x)\leq 0\}\neq \emptyset $. For each $x\in \mathcal{H}$, let  $g(x)$  be a chosen subgradient from the subdifferential set $\partial f(x):=\{g\in \mathcal{H}\colon f(y)\geq f(x)+\langle g,y-x\rangle \text{, for all }y\in \mathcal{H}\}$,  which, by \cite[Proposition 16.27]{BC17}, is nonempty. The  \emph{subgradient projection}
  \begin{equation}\label{ex:subProj:eq}
    P_{f}(x):=
    \begin{cases}
      x-\frac{f(x)}{\|g(x)\|^2}g(x), & \mbox{if } f(x)>0 \\
      x, & \mbox{otherwise}
    \end{cases}
  \end{equation}
  is a cutter and $\fix P_{f}=S(f,0)$; see, for example,  \cite[Proposition 2.3]{BC01} or  \cite[Corollary 4.2.6]{Ceg12}.
\end{example}
	
For a given $T\colon\mathcal H\rightarrow\mathcal H$ and $\alpha \in (0,\infty)$, the operator $U(x):=x +\alpha(T(x)-x)$ is called an $\alpha$-\textit{relaxation of} $T$. We call $\alpha$ a \textit{relaxation parameter}.  It is easy to see that $\fix T = \fix U$.  Usually, in connection with iterative methods, the relaxation parameter $\alpha$ is assumed to belong to the interval $(0,2]$.
	
\begin{lemma}[Projected Relaxation of a Cutter] \label{lem:ProjRelCutter}
  Let $T:\mathcal{H}\to\mathcal{H}$ be a cutter, $\alpha\in(0,2)$, and let $Q \subset \mathcal{H}$ be a closed and convex set such that $Q\cap\fix T \neq \emptyset$. Then the relaxation $U(x) := x+\alpha(T(x)-x)$ is $[(2-\alpha) /\alpha]$-SQNE. Moreover, the projected relaxation $P_Q U$ is $[(2-\alpha) /2]$-SQNE, where $\fix P_QU=Q\cap\fix T$.
\end{lemma}

\begin{proof}
	Combine,  for example,  \cite[Theorem 2.1.39]{Ceg12} and \cite[Corollary 2.1.47]{Ceg12} with the fact that $P_Q$ is $1$-SQNE.
\end{proof}
	
\begin{lemma} \label{lem:basic}
  Assume that $C_i=\fix T_i$ for given cutter operators $T_i\colon\mathcal H \to \mathcal H$, $i\in I$. Moreover, let  $\alpha\in (0,2)$,  let $\rho_i \colon \mathcal H \to (0,\infty)$, $i \in I$, and let $\lambda_i \in (0,1)$, $i \in I$, satisfy $\sum_{i\in I}\lambda_i = 1$. Define the operators $U_i(x) := x+\alpha\beta_i(x)(T_i(x)- x)$, $i\in I$, and the averaged operator $V(x):=\sum_{i\in I}\lambda_iU_i(x)$, $x \in \mathcal H$, where
  \begin{equation}\label{lem:basic:beta}
    \beta_i(x):=
    \begin{cases}\displaystyle
      \frac{\rho_i(x)+\|T_i(x)-x\|}{\|T_i(x)-x\|}, & \mbox{if } T_i(x) \neq x \\
      0, & \mbox{otherwise}.
    \end{cases}
  \end{equation}

  Assume that $C\cap Q \neq \emptyset$. Then $\fix V = C$ and $\fix P_Q V = C \cap Q$. Moreover, assume that there are $z\in C$ and $r>0$ such that $B(z,r)\subseteq C$. Then for all $x \notin C$ with $\rho(x):=\max_{i\in I}\rho_i(x)\leq r$, we have
  \begin{align}
    \|V(x)-z\|^2
    & \leq \|x - z\|^2 - \frac{2-\alpha}{\alpha} \sum_{i\in I} \lambda_i \|U_{i}(x) - x\|^2  \label{lem:basic:ineq1a}\\
		& \leq \|x - z\|^2 - \frac{2-\alpha}{\alpha}\|V(x) - x\|^2. \label{lem:basic:ineq1b}
  \end{align}
  If, in addition, $z \in Q$ (hence $z \in \interior(C) \cap Q$), then
	\begin{equation}\label{lem:basic:ineq2}
		\|P_Q(V(x)) - z\|^2 \leq \|x - z\|^2 - \frac{2-\alpha}{2}\|P_Q(V(x)) -
    x\|^2.
	\end{equation}
  Finally, when $x \in C$, then inequalities \eqref{lem:basic:ineq1a}--\eqref{lem:basic:ineq2} hold without any restrictions imposed on $\rho(x)$.
\end{lemma}

\begin{proof}
Both of the equalities, $\fix V = C$ and $\fix P_QV = C \cap Q$, follow from \cite[Lemma 3.2]{KRZ20}. The proof of \eqref{lem:basic:ineq1a} is similar to the proof of \cite[Proposition 4.5]{CZ14}. Indeed, let $x \notin C$ be such that $\rho(x) \leq r$ and assume that $U_i(x) \neq x$ for some $i \in I$. Then, by \cite[Lemma 3.1]{KRZ20}, we have
\begin{equation}\label{pr:lem:basic:Ui}
  \|U_{i}(x)-z\|^2\leq\|x-z\|^2-\frac{2-\alpha}{\alpha}\|x-U_{i}(x)\|^2.
\end{equation}
Using the properties of the inner product, we see that
\begin{align} \nonumber
  0 &\geq \|U_i(x)-z\|^2 - \|x-z\|^2 + \frac{2-\alpha}{\alpha}\|x-U_i(x)\|^2 \\ \nonumber
  & = \|(x-z) + (U_i(x)-x)\|^2 - \|x-z\|^2 + \frac{2-\alpha}{\alpha}\|x-U_i(x)\|^2 \\ \nonumber
  & = \|x-z\|^2 + \|U_i(x)-x\|^2 + 2\langle x-z, U_i(x)-x\rangle
  - \|x-z\|^2 + \frac{2-\alpha}{\alpha}\|x-U_i(x)\|^2 \\
  & = 2\langle x-z, U_i(x)-x\rangle + \frac 2 \alpha \|U_i(x)-x\|^2,
\end{align}
that is,
\begin{equation}\label{pr:lem:basic:Ui2}
  \alpha \langle z-x, U_i(x)-x\rangle \geq \|U_i(x)-x\|^2
\end{equation}
Note that inequality \eqref{pr:lem:basic:Ui2} holds for all $i \in I$, even when $U_i(x) = x$. Consequently,
\begin{align} \label{pr:lem:basic:V} \nonumber
	\|V(x)-z\|^2 &= \left\|(x-z)+\sum_{i\in I}\lambda_i(U_{i}(x)-x)\right\|^2\\ \nonumber
	& = \|x-z\|^2 + \left\|\sum_{i\in I}\lambda_i(U_{i}(x)-x)\right\|^2
  + 2\sum_{i\in }\lambda_i\langle x-z,U_i (x)-x\rangle\\ \nonumber
	& \leq \|x-z\|^2 + \sum_{i\in I}\lambda_i\|U_i(x)-x\|^2
  - \sum_{i\in I}\frac{2\lambda_i}{\alpha}\|U_i(x)-x\|^2\\ \nonumber
	& = \|x-z\|^2-\frac{2-\alpha}{\alpha}\sum_{i\in I}\lambda_i\|U_i(x)-x\|^2\\ \nonumber
	& \leq \|x-z\|^2-\frac{2-\alpha}{\alpha}\left\|\sum_{i\in I}\lambda_iU_{i}(x)-x\right\|^2\\
	& = \|x-z\|^2-\frac{2-\alpha}{\alpha}\|x-V(x)\|^2,
\end{align}
which proves \eqref{lem:basic:ineq1a}--\eqref{lem:basic:ineq1b}, as asserted.
		
We now show \eqref{lem:basic:ineq2}. Observe that, by assumption, the $r$-erosion $C^r = \{x \colon B(x,r)\subset C\}$ is nonempty, closed and convex. Thus the projection $P_{C^r}$ is well defined. Moreover, the auxiliary operator
\begin{equation}
  \tilde{V}(x) :=
  \begin{cases}
    V(x), & \text{if } x\in\mathcal H\setminus C \text{ and } \rho(x) \leq r,\\
    x+\alpha(P_{C}(x)-x), &\text{if } x\in \mathcal H \setminus C \text{ and } \rho(x) > r,\\
	  x+\alpha(P_{C^r}(x)-x), &\text{if } x\in C\setminus C^r,\\
	  x,&\text{if } x\in C^r.
	\end{cases}
\end{equation}
is $[\frac{2-\alpha}{\alpha}]$-SQNE and it satisfies $\fix\tilde{V}=C^r$. Indeed, by Lemma \ref{lem:ProjRelCutter}, applied separately to $P_C$ and $P_{C^r}$, we see that their $\alpha$-relaxations are $[\frac{2-\alpha}{\alpha}]$-SQNE. This, when combined with \eqref{pr:lem:basic:V}, proves that $\tilde V$ is $[\frac{2-\alpha}{\alpha}]$-SQNE. Again, by applying Lemma \ref{lem:ProjRelCutter}, this time to $\tilde{V}$ and $P_Q$, we see that $P_Q\tilde{V}$ is $[\frac{2-\alpha}{2}]$-SQNE. By the definition of $\tilde V$, we may now deduce inequality \eqref{lem:basic:ineq2}.
\end{proof}
	
\subsection{Regular sets}
Let $J$ be a nonempty set of indices, possibly uncountable. For example, $J$ can be defined as $\{1,2,\ldots\}$, $[0,1]$ or a combination of both. Moreover, let $\mathcal C := \{C_j\colon j\in J\}$ be a family of closed and convex sets $C_j \subset \mathcal H$ with nonempty intersection $C:=\bigcap_{j\in J}C_j$.  The following definition can be found in \cite{BB96}.
	
\begin{definition}\label{def:BR:set}
We say that the family $\mathcal C$ is
\begin{enumerate}[(i)]
  \item \emph{regular} over $S \subset \mathcal H$ if for every sequence $\{x_k\}_{k=0}^\infty \subset S$,
	\begin{equation}\label{eq:def:BR:set}
			\lim_{k\rightarrow\infty} \sup_{j\in J} d(x_k, C_j) = 0 \quad \Longrightarrow \quad
			\lim_{k\rightarrow\infty} d(x_k, C)=0;
	\end{equation}	
	\item \emph{linearly regular} over $S \subset \mathcal H$ if there is $\kappa > 0$ such that for every $x\in S$,
	\begin{equation}\label{eq:def:BLR:set}
			d(x, C) \leq \kappa \sup_{j\in J} d(x, C_j).
	\end{equation}
\end{enumerate}
If any of the above regularity conditions holds for every subset $S \subset \mathcal H$, then we simply omit the phrase ``over $S$". If the same condition holds when restricted to bounded subsets $S\subset\mathcal H$, then we precede the corresponding term with the adverb \textit{boundedly}.
\end{definition}

\begin{lemma}\label{lem:BLR}
  Assume that for some $j_0\in J$, we have $C_{j_0} \cap \interior\left(\bigcap_{j\in J\setminus\{j_0\}}C_j\right) \neq\emptyset.$ Then, for each bounded set $S\subset\mathcal{H}$, there is a number $\kappa_S\geq 1$ such that the inequality
	\begin{equation}\label{lem:BLR:ineq}
		d\left(x,\bigcap_{i\in I}C_i\right)\leq\kappa_S\sup_{i\in I}d(x,C_i)<\infty
	\end{equation}
  holds for all $x\in S$ and for all nonempty (possibly infinite) $I \subset J$. Moreover, for any ball $B(z_0,r)\subset\bigcap_{j\in J\setminus\{j_0\}}C_j$, where $z_0\in C_{j_0}$, one can use
	\begin{equation}\label{lem:BLR:kappa}
		\kappa_S := 1+\frac{2\sup_{x\in S}\|x-z_0\|}{r}<\infty,
	\end{equation}
	which, in particular, is independent of $I$. As a consequence, every subfamily $\{C_i \colon i\in I\}$ is boundedly linearly regular.
\end{lemma}
	
\begin{proof}
Let $I\subset J$ be nonempty. If $j_0\in I$, then set $j=j_0$, otherwise choose any $j\in J$. Define $C_I:=\bigcap_{i\in I} C_i$
and observe that for each point $x\in\mathcal{H}$, we have $\phi_I(x):=\sup_{i\in I}d(x,C_i)\leq d(x,C_I)<\infty.$ Moreover, since $S$ is bounded, we also have $\sup_{x\in S}\phi_I(x)\leq\sup_{x\in S}d(x,C_i)\leq\sup_{x\in S}\|x-z_0\|=:R<\infty$, where $B(z_0,r)$ is chosen as above.
		
\underline{\emph{Step 1.}} We show that for any $y\in\mathcal{H}$ such that $\phi_{I\setminus\{j\}}(y)\leq\varepsilon$, we have
\begin{equation}\label{pr:BLR:w}
	w:=\frac{\varepsilon}{\varepsilon+r}z_0+\frac{r}{\varepsilon+r}y\ \in\bigcap_{i\in I\setminus\{j\}}C_i.
\end{equation}
Indeed, for any $i\in I\setminus\{j\}$, we have
\begin{equation}\label{pr:BLR:convexw}
  w=\frac{\varepsilon}{\varepsilon+r}\underbrace{\left(z_0+\frac{r}{\varepsilon}(y-P_{C_i}(y))\right)}_{:=z}+\frac{r}{\varepsilon+r}P_{C_i}(y).
\end{equation}
Observe that
\begin{equation}
	\|z-z_0\|=\frac{r}{\varepsilon}d(y,C_i)\leq\frac{r}{\varepsilon}\phi_{I\setminus\{j\}}(y)\leq r
\end{equation}
and hence $z\in B(z_0,r)$. In particular, $z\in C_i$, which is convex, and thus using \eqref{pr:BLR:convexw}, we conclude that $w$ is a convex combination of two points in $C_i$, thus $w \in C_i$. Since the index $i$ has been arbitrarily chosen, this finishes the proof of Step 1.
		
\underline{\emph{Step 2.}} Next we show that inequality \eqref{lem:BLR:ineq} holds with $\kappa_S$ defined by \eqref{lem:BLR:kappa}. Indeed, let $x\in S$ and define $y:=P_{C_j}(x)$. Setting $\varepsilon=2\phi_I(x)$, for all $i\in I\setminus\{j\}$, we obtain
\begin{equation}
  d(y,C_i) = \|y-P_{C_i}(y)\| \leq \|y-P_{C_i}(x)\| \leq \|y-x\|+\|x-P_{C_i}(x)\| \leq \frac{\varepsilon}{2} + \frac{\varepsilon}{2} = \varepsilon.
\end{equation}
Consequently, $\phi_{I\setminus\{j\}}(y) \leq \varepsilon$ and by applying Step 1 to the point $w$ defined by \eqref{pr:BLR:w} with the above point $y$, we obtain that $w\in\bigcap_{i\in I\setminus\{j\}}C_i$. Since both $z_0$ and $w$ are in $C_j$, which is convex, we get that $w\in C_j$. Consequently, $w\in C_I$ and thus
\begin{align}
		d(x,C_I) & \leq \|x-w\|\leq\|x-y\|+\|y-w\|\nonumber\\
		& = d(x,C_j)+\frac{2\phi_I(x)}{2\phi_I(x)+r}\|P_{C_j}(x)-z_0\|\nonumber\\
		& \leq \phi_I(x)+\frac{2\phi_I(x)}{r}\|x-z_0\|\nonumber\\
		& \leq \phi_I(x)\left(1+\frac{2R}{r}\right).
\end{align}
This completes the proof of Step 2 and of the lemma itself.
\end{proof}

\begin{remark} \label{rem:interiorAndBLR}
  The proof of Lemma \ref{lem:BLR} is a slight modification of the proof given in \cite[Lemma 5]{GPR67}, where the assumption $C_{j_0} \cap \interior\left(\bigcap_{j\in J\setminus\{j_0\}}C_j\right) \neq\emptyset$ appeared for the first time. This assumption can also be found, for example, in \cite[Corollary 3.15]{BB96} with a finite number of sets, and in \cite[Theorem 3.1]{ZNL18}, with an infinite number of sets. For other examples of regular families of sets, we refer the reader to \cite{BB96}.
\end{remark}

\subsection{Regular operators}
 Following \cite{BNPH15} and \cite{CRZ18}, we now  present the definition of regular operators.

\begin{definition}\label{def:BR:oper}
  Let $T\colon\mathcal H\to\mathcal H$ be an operator with a fixed point, that is, $\fix T \neq \emptyset$ and let $S\subset\mathcal H$ be nonempty. We say that the operator $T$ is
	\begin{enumerate}[(i)]
		\item \textit{weakly regular} over $S$ if for any sequence $\{x_{k}\}_{k=0}^\infty \subset S$ and $x_\infty\in \mathcal H$,
		  \begin{equation}\label{eq:def:DC}
        \left .
			  \begin{array}{l}
          x_{k}\rightharpoonup x_\infty \\ T (x_k)-x_k\rightarrow 0
			  \end{array}
			  \right \}
        \quad\Longrightarrow\quad x_\infty\in \fix T;
	    \end{equation}			
		\item \textit{regular} over $S$ if for any sequence $\{x_{k}\}_{k=0}^\infty \subset S$,
			\begin{equation} \label{eq:def:BR:oper}
			 \lim_{k\rightarrow\infty}\| T(x_k)-x_k\| =0\quad\Longrightarrow\quad \lim_{k\rightarrow\infty}d(x_k,\fix T)=0;
			\end{equation}			
		\item \textit{linearly regular} over $S$ if there is $\delta_S>0$ such that for every $x\in S$,
			\begin{equation} \label{eq:def:BLR:oper}
        \delta_S d(x,\fix T)\leq \|T(x)-x\|.
			\end{equation}
	\end{enumerate}
  If any of the above regularity conditions holds for every subset $S\subset\mathcal H$, then we simply omit the phrase ``over $S$". If the same condition holds when restricted to bounded subsets $S\subset\mathcal H$, then we precede the corresponding term with the adverb \textit{boundedly}. Since there is no need to distinguish between boundedly weakly and weakly regular operators, we call both of them weakly regular.
\end{definition}

 Obviously, the operator $T$ is weakly regular if and only if $T - I$ is demiclosed at zero.  Note here that (iii) implies (ii) which implies (i). For a more extensive overview see, for example, \cite{CRZ18}. Other examples can be found in \cite{BNPH15, BLT15, CET20, KRZ17}.

\subsection{Quasi-Fej\'er monotone sequences}
Following Combettes \cite{Com01}, we recall the following definitions.	
\begin{definition}
Let $C\subset\mathcal H$ be a nonempty, closed and convex set. We say that a sequence $\{x_k\}_{k=0}^\infty$ in $\mathcal{H}$ is:
\begin{enumerate}[(i)]
	\item \emph{Fej\'er monotone} (FM) with respect to $C$ if for all
    $z\in C$ and every integer $k = 0,1,2,\ldots$, we have
    $\|x_{k+1}-z\| \leq \|x_k-z\|$.

  \item \emph{Quasi-Fej\'er monotone of type I} (QF1) with respect to $C$
    if there is a sequence $\{\varepsilon_k\}_{k=0}^\infty \subset (0,\infty)$ satisfying $\sum_{k=0}^\infty \varepsilon_k < \infty$, such that for all $z\in C$ and every integer $k = 0,1,2,\ldots$, we have $\|x_{k+1}-z\|\leq\|x_k-z\|+\varepsilon_k$.
	
	\item \emph{Quasi-Fej\'er monotone of type II} (QF2) with respect to $C$
    if there is a sequence $\{\varepsilon_k\}_{k=0}^\infty \subset (0,\infty)$ satisfying $\sum_{k=0}^\infty \varepsilon_k < \infty$, such that for all $z\in C$ and every integer $k = 0,1,2,\ldots$, we have $\|x_{k+1}-z\|^2\leq\|x_k-z\|^2+\varepsilon_k$.
\end{enumerate}
\end{definition}
		
\begin{theorem}\label{thm:QFM}
  Let $\{x_k\}_{k=0}^\infty\subset\mathcal{H}$ be FM, QF1 or QF2 with respect to a nonempty, closed and convex set $C\subset \mathcal H$. Then the following statements hold:
	\begin{enumerate}[(i)]
		\item The sequence $\{x_k\}_{k=0}^\infty$ is bounded.
		\item The sequence $\{x_k\}_{k=0}^\infty$ converges weakly to some
      point $x_\infty \in C$ if and only if all its weak cluster points lie in $C$.
		\item The sequence $\{x_k\}_{k=0}^\infty$ converges strongly to some
      point $x_\infty\in C$ if and only if $d(x_k,C) \to 0$ as ${k\to\infty}$.
	\end{enumerate}
\end{theorem}

\begin{proof}
  The first part follows from \cite[Proposition 3.38]{Com01}, the second part follows from \cite[Theorem 3.8]{Com01}, and the third part follows from \cite[Theorem 3.11]{Com01} and \cite[Proposition 3.5]{Com01}.
\end{proof}
 The following lemma is a slight modification of \cite[Proposition 4.2]{Com01}.
\begin{lemma}\label{lem:QFM}
Let $\{W_k\}_{k=0}^\infty$ be a sequence of cutters such that $C:=\bigcap_{k=0}^\infty\fix W_k\neq\emptyset$ and let $Q\subset\mathcal{H}$ be closed and convex such that $C\cap Q\neq\emptyset$. Moreover, let $\{\alpha_k\}_{k=0}^\infty\in[\varepsilon,2-\varepsilon]$ be a sequence of relaxations and let $\{e_k\}_{k=0}^\infty \subset \mathcal{H}$ be such that $\sum_{k=0}^\infty\|e_k\|<\infty$. Then the sequence $\{x_k\}_{k=0}^\infty$ defined by
\begin{equation} \label{lem:QFM:xk}
	x_0\in\mathcal H,\quad x_{k+1}:=P_Q\left(x_k+\alpha_k\big(W_k(x_k)-x_k+e_k\big)\right),
\end{equation}
is QF1 with respect to $C\cap Q$, where $\varepsilon_k=\alpha_k\|e_k\|$. Morever,
\begin{equation} \label{lem:QFM:sum}
  \sum_{k=0}^\infty\|W_k(x_k)-x_k\|^2<\infty
  \quad \text{and} \quad
  \sum_{k=0}^\infty\|x_{k+1}-x_k\|^2<\infty.
\end{equation}

\end{lemma}

\begin{proof}
We first show that $\{x_k\}_{k=0}^\infty$ is QF1 with respect to $C\cap Q$. To this end, let $z \in C \cap Q$. Note that $P_Q$ is nonexpansive. Moreover, using Lemma \ref{lem:ProjRelCutter} applied to $W_k$, we see that its $\alpha_k$ relaxation is $\frac{2-\alpha_k}{\alpha_k}$-SQNE, hence QNE. Thus, by the triangle inequality, we have,
\begin{align} \label{pr:QFM:ineq1}\nonumber
  \|x_{k+1}-z\| &= \|P_Q(x_k+\alpha_k(W_k(x_k)-x_k+e_k))-z\|\\ \nonumber
  & \leq \|x_k+\alpha_k(W_k(x_k)-x_k)-z\|+\alpha_k\|e_k\|\\
  & \leq \|x_k-z\|+\varepsilon_k,
\end{align}
which shows that $\{x_k\}_{k=0}^\infty$ is QF1, as asserted.

Next, we show that the first series in \eqref{lem:QFM:sum} is summable. By inductively applying \eqref{pr:QFM:ineq1}, we see that
\begin{equation}\label{}
  \|x_k-z\| \leq \|x_0-z\|+\sum_{k=0}^\infty \varepsilon_k=:R<\infty.
\end{equation}
Set $y_k:=x_k+\alpha_k(W_k(x_k)-x_k+e_k)$. By using the facts that $P_Q$ is 1-SQNE and that the $\alpha_k$-relaxation of $W_k$ is $\frac{2-\alpha_k}{\alpha_k}$-SQNE, we see that
\begin{align} \nonumber
	\|x_{k+1}-z\|^2 & = \|P_Q(y_k)-z\|^2 \leq \|y_k-z\|^2 - \|P_Q(y_k)-y_k\|^2\\ \nonumber
  & = \|(x_k+\alpha_k(W_k(x_k)-x_k)-z) + \alpha_k e_k\|^2 - \|P_Q(y_k)-y_k\|^2\\ \nonumber
	&\leq \|x_k+\alpha_k(W_k(x_k)-x_k)-z\|^2
    + 2\alpha_k\|e_k\| \cdot \|x_k+\alpha_k(W_k(x_k)-x_k)-z\|\\ \nonumber
  & \quad + \alpha_k^2\|e_k\|^2 \nonumber - \|P_Q(y_k)-y_k\|^2\\ \nonumber
  & \leq \|x_k-z\|^2 -(2-\alpha_k)\alpha_k\|W_k(x_k)-x_k\|^2
    + 2\varepsilon_k \|x_k-z\| \\ \nonumber
  & \quad + \varepsilon_k^2 - \|P_Q(y_k)-y_k\|^2\\
  & \leq \|x_k-z\|^2 -\varepsilon^2\|W_k(x_k)-x_k\|^2 + 2 R \varepsilon_k + \varepsilon_k^2 - \|P_Q(y_k)-y_k\|^2.
\end{align}
Thus
\begin{equation}\label{}
  \varepsilon^2 \sum_{k=0}^\infty \|W_k(x_k)-x_k\|^2 +  \sum_{k=0}^\infty \|P_Q(y_k)-y_k\|^2
  \leq 2 R \sum_{k=0}^{\infty} \varepsilon_k + \sum_{k=0}^{\infty} \varepsilon_k^2 < \infty,
\end{equation}
which confirms the first inequality of \eqref{lem:QFM:sum}. Furthermore, we have
\begin{align} \nonumber
	\|x_{k+1}-x_k\|^2 &= \|P_Q(y_k)-x_k\|^2 \leq \|y_k-x_k\|^2 = \|x_k+\alpha_k(W_k(x_k)-x_k+e_k)-x_k\|^2\\ \nonumber
	& \leq \alpha_k^2(\|W_k(x_k)-x_k\|^2+\|e_k\|^2+2\|W_k(x_k)-x_k\|\cdot\|e_k\|)\\
	& \leq (2-\varepsilon)^2(\|W_k(x_k)-x_k\|^2+\|e_k\|^2 + R'\|e_k\|),
\end{align}
where $R':=\sup_{k}\|W_k(x^k)-x^k\| < \infty$. This implies the second inequality of \eqref{lem:QFM:sum}.
\end{proof}

\section{Finite Convergence} \label{sec:finiteConv}

In this section we proclaim and establish our main result regarding finite convergence.

\begin{theorem}\label{thm:main}
Let $\{x_k\}_{k=0}^\infty$ be a sequence defined by \eqref{int:xk}--\eqref{int:betak}. Assume that
\begin{enumerate}[(i)]
  \item $\interior(C)\cap Q \neq \emptyset$.
  \item $r_k \to 0$ monotonically, but with a rate slower than any geometric sequence; see \eqref{int:STAG}.
  \item For each bounded subset $S \subset \mathcal H$, there are $\delta, \Delta >0$ s.t. $\delta \leq \varphi_i(x) \leq \Delta$ for all $x \in S$.
  \item There is $\varepsilon > 0$ such that $\lambda_{i,k} \in [\varepsilon, 1]$ and $\alpha_k \in [\varepsilon, 2-\varepsilon]$ for all $i, k$.
  \item $\{I_k\}_{k=0}^\infty$ is $s$-intermittent for some integer $s \geq 1$, that is,  $I=I_k\cup\ldots\cup I_{k+s}$ for all $k$.
  \item $T_i$ are boundedly linearly regular, $i\in I$.
\end{enumerate}
If  the sequence $\{x_k\}_{k=0}^\infty$ is bounded (see Example \ref{ex:phi}),  then $x_k\in C\cap Q$ for some $k$.
\end{theorem}

\begin{proof}
Let $\{x_k\}_{k=0}^\infty$ be a bounded trajectory defined by \eqref{int:xk}--\eqref{int:betak} and suppose to the contrary that $\{x_k\}_{k=0}^\infty \subset \mathcal H\setminus (C \cap Q)$. We divide the rest of the proof into a few technical steps, where in the last step, we arrive at a contradiction using assumption (ii). Before doing so, we set up some notations and present basic properties of the generated trajectory.

Following the notation of Lemma \ref{lem:basic}, for each $k$ and $i$, set
\begin{equation}\label{}
  V_k(x) := \sum_{i\in I_k} \lambda_{i,k}U_{i,k}(x),
  \quad \text{where} \quad
  U_{i,k}(x) := x + \alpha_k \beta_{i,k}(x)(T_i(x) - x),\quad x \in \mathcal H.
\end{equation}
Note that we can write $x_{k+1} = P_Q(V_k(x_k))$ for short. Let $S$ be a bounded subset of $\mathcal H$ containing $\{x_k\}_{k=0}^\infty$. By assumptions (iii) and (iv), there are $\delta,\Delta>0$, both depending on $S$, such that for all $x \in S$ and for all $i\in I$, we have
\begin{equation}\label{pr:main:delta}
  \delta d(x, C_i) \leq \|T_i(x) - x\|
  \qquad \text{and} \qquad
  \delta \leq \varphi_i(x) \leq \Delta.
\end{equation}
This holds, in particular, for all $x = x_k$, $k= 0, 1 ,2, \ldots$. In addition, by assumption (i), there are $r>0$ and $z\in Q$ such that $B(z,r) \subset C$. Moreover, by condition (ii), we see that $\varepsilon_k := \frac{r_k}{\delta} \leq r$ for all large enough $k\geq K \geq 0$. Without loss of generality, we may assume that $\varepsilon_K= r$. In addition, assumption (i) guarantees a uniform linear regularity condition over $S$ for any subfamily of
\begin{equation}\label{}
  \mathcal D := \{C_i \colon i \in I\}
  \cup \{C^\varepsilon \colon 0 \leq \varepsilon \leq  r\}
  \cup \{Q\},
\end{equation}
as in Lemma \ref{lem:BLR}. Recall here that the $\varepsilon$-erosion $C^{\varepsilon} = \{ x \in C \colon B(x,\varepsilon) \subset C\}$; compare with \eqref{def:erosion}. Indeed, by eventually decreasing $r$, we may ensure that $\interior(C^{r}) \cap Q \neq \emptyset$. This leads to $\interior \big(\bigcap (\mathcal D \setminus \{Q\}) \big) \cap Q \neq \emptyset$. Consequently, by Lemma \ref{lem:BLR}, there is $\kappa = \kappa_S > 0$, such that for each nonempty subfamily $\mathcal D'$ of $\mathcal D$ and for each $x \in S$, we have
\begin{equation}\label{pr:main:BLR}
  d\left(x, \bigcap \mathcal D'\right) \leq \kappa \sup_{D \in \mathcal D'} d(x, D).
\end{equation}
		
\underline{\emph{Step 1.}} We claim that for all $k \geq K$, we have
\begin{equation}\label{pr:main:step1}
  d^2(x_k, C ) \leq \frac{2 s \kappa^2}{\delta^2 \varepsilon^3}
  \left( d^2(x_k, C^{\varepsilon_k} \cap Q) - d^2(x_{k+s}, C^{\varepsilon_k} \cap Q)\right).
\end{equation}
Let $n\geq k\geq K$. Note that by the monotonicity of $\{\varepsilon_k\}_{k=0}^\infty$, we have $C^{\varepsilon_k} \cap Q \subset C^{\varepsilon_n} \cap Q$. Therefore, by Lemma \ref{lem:basic} (inequality \eqref{lem:basic:ineq2}), we see that
\begin{equation}\label{pr:main:xnFejer}
  \|x_{n+1}-x_n\|^2 \leq \frac{2}{2-\alpha_n}\big(\|x_n-z\|^2 - \|x_{n+1}-z\|^2\big)
  \leq \frac{2}{\varepsilon}\big(\|x_n-z\|^2 - \|x_{n+1}-z\|^2\big)
\end{equation}
for all $z\in C^{\varepsilon_k} \cap Q$. Note that \eqref{pr:main:xnFejer} also holds trivially when $x_{n+1} = x_n$. This implies that the tail $\{x_n\}_{n=k}^\infty$ is Fej\'er monotone towards $C^{\varepsilon_k} \cap Q$.

Let $i \in I$ and for each $k \geq K$, let $m_k = m_k(i) \in \{k,k+1,\ldots,k+s-1\}$ be the smallest integer for which $i \in I_{m_k}$. By the triangle inequality and the properties of the metric projection, we have
\begin{equation}\label{}
  d(x_k, C_i) \leq \|x_k - P_{C_i}(x_{m_k})\| \leq \sum_{n=k}^{m_k-1}\|x_{n+1} - x_n\| + d(x_{m_k},C_i)
\end{equation}
which, when combined with the Cauchy-Schwarz inequality, leads to
\begin{equation}\label{pr:dxkCi}
  d^2(x_k, C_i) \leq s \left( \sum_{n=k}^{m_k-1} \|x_{n+1} - x_n\|^2 + d^2(x_{m_k},C_i)\right).
\end{equation}
Note here that we use the Cauchy-Schwarz inequality for real numbers, that is, $(\sum_{i=1}^{m}a_i)^2 \leq m \cdot \sum_{i=1}^{m}a_i^2$. By setting $z_k := P_{C^{\varepsilon_k} \cap Q}(x_k)$ in \eqref{pr:main:xnFejer}, we see that the first summand from \eqref{pr:dxkCi} can be estimated by
\begin{align}\label{pr:sumxn}
  \sum_{n=k}^{m_k-1} \|x_{n+1} - x_n\|^2 \nonumber
  & \leq \frac{2}{\varepsilon}\big(\|x_k-z_k\|^2 - \|x_{m_k}-z_k\|^2\big) \\ \nonumber
  & \leq \frac{2}{\varepsilon}\big(\|x_k-z_k\|^2 - \|x_{m_k}-z_{m_k}\|^2\big) \\
  & = \frac{2}{\varepsilon}\big( d^2(x_k, C^{\varepsilon_k} \cap Q) - d^2(x_{m_k}, C^{\varepsilon_k} \cap Q) \big).
\end{align}
By \eqref{pr:main:delta}, by using \eqref{lem:basic:ineq1a} with $z_k' := P_{C^{\varepsilon_k} \cap Q}(x_{m_k})$ and by the above mentioned Fej\'{e}r monotonicity of the tail $\{x_n\}_{n=k}^\infty$, we can estimate the second summand in \eqref{pr:dxkCi} by
\begin{align}\label{pr:dxmkCi} \nonumber
  d^2(x_{m_k}, C_i)
  & \leq \frac 1 {\delta^2} \|T_i(x_{m_k}) - x_{m_k}\|^2
  \leq \frac 1 {\delta^2 \alpha_{m_k}^2} \|U_{i,m_k}(x_{m_k}) - x_{m_k}\|^2 \\ \nonumber
  & \leq \frac 1 {\delta^2 \alpha_{m_k}^2 \lambda_{i, m_k}} \big(\|x_{m_k}-z_{k}'\|^2 - \|V_{m_k}(x_{m_k})-z_{k}'\|^2\big) \\ \nonumber
  & \leq \frac 1 {\delta^2 \varepsilon^3} \big(\|x_{m_k}-z_{k}'\|^2 - \|P_Q(V_{m_k}(x_{m_k}))-z_{k}'\|^2\big) \\ \nonumber
  & \leq \frac 1 {\delta^2 \varepsilon^3} \big(\|x_{m_k}-z_{k}'\|^2 - \|x_{k+s}-z_{k}'\|^2\big) \\ \nonumber
  & \leq \frac 1 {\delta^2 \varepsilon^3} \big(\|x_{m_k}-z_{k}'\|^2 - \|x_{k+s}-z_{k+s}\|^2\big) \\
  & = \frac 1 {\delta^2 \varepsilon^3} \big( d^2(x_{m_k}, C^{\varepsilon_k} \cap Q) - d^2(x_{k+s}, C^{\varepsilon_k} \cap Q) \big).
\end{align}
Consequently, by \eqref{pr:dxkCi}--\eqref{pr:dxmkCi} and by the arbitrariness of $i \in I$, we have
\begin{equation}\label{pr:maxdxkCi}
  \max_{i \in I}d^2(x_k, C_i) \leq \frac{2s}{\delta^2 \varepsilon^3}
  \left( d^2(x_k, C^{\varepsilon_k} \cap Q) - d^2(x_{k+s}, C^{\varepsilon_k} \cap Q)\right).
\end{equation}
This, when combined with \eqref{pr:main:BLR} applied to the family $\mathcal D' = \{C_i \colon i \in I\}$, shows \eqref{pr:main:step1}, as asserted.

\underline{\emph{Step 2.}} We show that for all $k \geq K$ such that $x_{k+1} \neq x_k$, we have
\begin{equation}\label{pr:main:step2}
  \varepsilon_k^2 \leq \frac{\Delta^2}{\varepsilon^3 \delta^2}
  \left( d^2(x_k, C^{\varepsilon_k} \cap Q) - d^2(x_{k+s}, C^{\varepsilon_k} \cap Q)\right).
\end{equation}
Indeed, let $i\in I_k$ be such that $T_i(x_k) \neq x_k$. Then, by the definition of $U_{i,k}$, we have
\begin{equation}\label{}
  \|U_{i,k}(x_k)-x_k\| = \alpha_k \beta_{i,k}(x_k)\|T_i(x_k)-x_k\|
  \geq \varepsilon \frac{r_k}{\varphi_i(x_k)}
  \geq  \frac{\varepsilon \delta}{\Delta} \varepsilon_k.
\end{equation}
By using an argument similar to \eqref{pr:dxmkCi}, we arrive at
\begin{align}\label{}
  \varepsilon_k^2
  & \leq \frac{\Delta^2}{\varepsilon^2 \delta^2} \|U_{i,k}(x_k)-x_k\|^2
  \leq \frac{\Delta^2}{\varepsilon^2 \delta^2 \lambda_{i,k}}  \big(\|x_k-z_k\|^2 - \|V_k(x_k)-z_k\|^2\big) \nonumber \\
  & \leq \frac{\Delta^2}{\varepsilon^3 \delta^2} \big(\|x_k-z_k\|^2 - \|P_Q(V_k(x_k))-z_k\|^2\big) \nonumber \\
  & = \frac{\Delta^2}{\varepsilon^3 \delta^2} \big(\|x_k-z_k\|^2 - \|x_{k+s}-z_k\|^2\big) \nonumber \\
  & = \frac{\Delta^2}{\varepsilon^3 \delta^2} \big(\|x_k-z_k\|^2 - \|x_{k+s}-z_{k+s}\|^2\big) \nonumber \\
  & \leq \frac{\Delta^2}{\varepsilon^3 \delta^2} \big(d^2(x_k,C^{\varepsilon_k} \cap Q) - d^2(x_{k+s},C^{\varepsilon_k} \cap Q)\big),
\end{align}
which proves \eqref{pr:main:step2}.

\underline{\emph{Step 3.}} We show that for all $k \geq K$, such that $x_{k+1} \neq x_k$, we have
\begin{equation}\label{pr:main:step3}
	d(x_{k+s},Q\cap C^{\varepsilon_k})\leq q \cdot d(x_k,Q\cap C^{\varepsilon_k}),
\end{equation}
where
\begin{equation}\label{pr:main:qAndM}
  q := \sqrt{1 - \frac {\delta^2 \varepsilon^3}{2 s \kappa^4 \Delta^2 M^2}}\in (0,1)
  \quad \text{and} \quad
  M := \frac{d(x_K, C^r \cap Q)}{r} \geq 1.
\end{equation}
Indeed, by \eqref{pr:main:BLR} applied to the subfamily $\mathcal D_k':=\{C^{\varepsilon_k}, Q\}$, we get
\begin{equation}\label{}
  d^2(x_k, C^{\varepsilon_k} \cap Q) \leq \kappa^2 \max \{d^2(x_k, C^{\varepsilon_k}), d^2(x_k,Q) \} = \kappa^2 d^2(x_k, C^{\varepsilon_k}).
\end{equation}
On the other hand, by the Fej\'er monotonicity of $\{x_n\}_{n=K}^\infty$ towards $C^r \cap Q$, we have $d(x_k,C^r)/r \leq M$.
This, when combined with \eqref{eq:erosion}, the Cauchy-Schwarz inequality, \eqref{pr:main:step1} and \eqref{pr:main:step2}, yields
\begin{align}\label{}
  d^2(x_k, C^{\varepsilon_k})
  & \leq (d(x_k, C) + \varepsilon_k)^2 \frac{d^2(x, C^r \cap Q)}{(d(x_k,C) + r)^2}
  \leq 2 M^2 (d^2(x_k, C) + \varepsilon_k^2)  \nonumber \\
  & \leq \frac{2 s \kappa^2 \Delta^2 M^2}{\delta^2 \varepsilon^3}
  \left( d^2(x_k, C^{\varepsilon_k} \cap Q) - d^2(x_{k+s}, C^{\varepsilon_k} \cap Q)\right)
\end{align}
and consequently, we have
\begin{equation}\label{}
  d^2(x_k, C^{\varepsilon_k} \cap Q) \leq \frac{2 s \kappa^4 \Delta^2 M^2}{\delta^2 \varepsilon^3}
  \left( d^2(x_k, C^{\varepsilon_k} \cap Q) - d^2(x_{k+s}, C^{\varepsilon_k} \cap Q)\right).
\end{equation}
By simply rearranging the terms, we arrive at \eqref{pr:main:step3}.

\underline{\emph{Step 4.}} We arrive at a contradiction with the assumption that $\{x_k\}_{k=0}^\infty\subset\mathcal H\setminus C \cap Q$ using condition (ii). Indeed, since the control is $s$-intermittent and since $\fix P_Q V_k=Q \cap \bigcap_{i\in I_k}C_i$, there must be a sequence $\{k_n\}_{n=0}^\infty$, where each $k_n\in\{K+2ns,\ldots,K+(2n+1)s\}$ is the smallest number such that $x_{k_n}\neq x_{k_n+1}$. The monotonicity of $\{\varepsilon_k\}_{k=0}^\infty$ and the above-mentioned Fej\'er monotonicity of the sequence $\{x_n\}_{n=k}^\infty$ towards $C^{\varepsilon_k} \cap Q$, which holds for all $k \geq K$, when combined with \eqref{pr:main:step3}, give us
\begin{equation}\label{}
  d(x_{k_{n+1}}, C^{\varepsilon_{k_{n+1}}} \cap Q)
  \leq d(x_{k_{n+1}}, C^{\varepsilon_{k_n}} \cap Q)
  \leq d(x_{k_n+s}, C^{\varepsilon_{k_n}} \cap Q)
  \leq q \cdot d(x_{k_n}, C^{\varepsilon_{k_n}} \cap Q).
\end{equation}
By \eqref{pr:main:step3}--\eqref{pr:main:qAndM} and by using induction, we arrive at
\begin{equation}\label{linear}
  \frac{\varepsilon \delta }{\Delta} \varepsilon_{k_n}
  \leq d(x_{k_n}, C^{\varepsilon_{k_n}} \cap Q)
  \leq q^n \cdot d(x_{k_0}, C^{\varepsilon_{k_0}} \cap Q) \leq M r q^n,
\end{equation}
which, after rearranging, leads to $\varepsilon_{k_n} \leq c q^n$, where $c = \frac{Mr\Delta}{\varepsilon \delta}$. On the other hand, by condition (ii), there is a natural number $K'$ such that for all $k \geq K'$, we have
\begin{equation}\label{}
  \varepsilon_k > (c q^{\frac{-K}{2s}}) \cdot q^{\frac{k}{2s}},
\end{equation}
which, when applied to large enough $k = k_n$, yields
\begin{equation}\label{}
  \varepsilon_{k_n}
  > (c q^{\frac{-K}{2s}}) \cdot q^{\frac{k_n}{2s}}
  \geq (c q^{\frac{-K}{2s}}) \cdot q^{\frac{K + 2ns}{2s}}
  = c q^n.
\end{equation}
Thus we have arrived at a contradiction and consequently we must have $x_k\in C\cap Q$ for some $k$. This completes the proof.
\end{proof}

It is clear that the boundedness of the trajectories $\{x_k\}_{k=0}^\infty$ defined in Theorem \ref{thm:main} is guaranteed when the set $Q$ is bounded. However, there are two situations in which boundedness of $Q$ is not required and where the boundedness of $\{x_k\}_{k=0}^\infty$ still holds. We emphasize here that assumption (i) of Theorem \ref{thm:main}, or its variant (i') described below, and the assumption that $r_k \to 0$ are crucial.

\begin{example}[Bounded Trajectories] \label{ex:phi} ~
  \begin{enumerate}[(a)]
    \item Following \cite[Example 4.8]{KRZ20}, for each $x \in \mathcal H$, define
    \begin{equation}\label{ex:phi:eq2}
      \varphi_i(x) := 1.
    \end{equation}
  Obviously, $\varphi_i(x)$ satisfies (iii). Furthermore, assumptions (i), (ii) and (iii) allow us to apply Lemma \ref{lem:basic}. In particular, in view of \eqref{lem:basic:ineq2}, the sequence $\{x_k\}_{k=0}^\infty$ is bounded. Consequently, $x_k\in C\cap Q$ for some $k$.

  \item Following \cite[Example 4.9]{KRZ20}, assume that $C_i=\{x\colon  f_i(x)\leq 0\}$ for some convex and lower semicontinuous functions $f_i\colon\mathcal H \to \mathbb R$, $i\in I$. Let $T_i := P_{f_i}$ and let $g_i\colon \mathcal H \to \mathcal H$ be the associated subgradient mapping. For each $x\in \mathcal H$, define
  \begin{equation}\label{ex:phi:eq2}
    \varphi_i(x):=
    \begin{cases}
      \|g_i(x)\|, & \mbox{if } f_i(x)>0 \\
      1, & \mbox{otherwise}.
    \end{cases}
  \end{equation}
  Assume that (i') $f(z):= \max_{i\in I} f_i(z) <0$ for some $z\in Q$; (ii') = (ii); and (iii') $\partial f_i (S)$ is bounded for bounded subsets $S\subset \mathcal H$, $i\in I$. By using (i') and (ii'), one can show that the sequence $\{x_k\}_{k=0}^\infty$ is bounded. Moreover, assumptions (i'), (ii') and (iii') imply conditions (i),(ii) and (iii); see \cite[Example 4.9]{KRZ20}. Furthermore, conditions (i') and (iii') imply condition (vi) as $P_{f_i}$ is boundedly linearly regular (see \cite[Example 2.11]{CRZ20}). Consequently, $x_k\in C\cap Q$ for some $k$.
  \end{enumerate}
\end{example}

\begin{remark}[see also Remark \ref{rem:interiorAndBLR}]
\begin{enumerate} [(a)]
  \item Condition (i) of Theorem 3.1 that $\interior (C) \cap Q \neq \emptyset$ appears in many other works concerning finitely convergent algorithms with overrelaxations; see for example,  \cite{BWWX15, Cro04, KRZ20, Pol01}. Moreover, even a much stronger Slater's condition appears in \cite{CCP11, DePI88, IM86}. Furthermore, there are other algorithms which also require variation of condition (i) in order to guarantee finite convergence; see, for example, \cite{BDNP16, Fuk82, IM87, Pan14}. Thus, it appears that in a theoretical study condition (i) is a rather standard assumption in obtaining the finite convergence property.

  \item On the other hand, the requirement regarding the nonempty interior of one of the constraints is quite strong, in particular, when the space $\mathcal H$ is infinite-dimensional. An example of such a set concerning inequality constraints can be found in \cite[Section 4]{BM21}. A nontrivial sufficient condition for nonempty interior is given in \cite[Theorem 1]{KP89} in the setting of an ordered Banach space. Another nontrivial example (related to the constrained infinite-dimensional Mayer problem) can be found in the proof of \cite[Theorem 4.2]{FMM18}.

  \item There are many other notions that generalize the concept of an interior in convex optimization; see, for example, \cite{BC17, BC12, Z02, Z15}. Thus it might be of interest to extend the result of Theorem \ref{thm:main} with the use of another interiority notion. We leave this part as an open problem for further research.
\end{enumerate}
\end{remark}

\section{Asymptotic Convergence} \label{sec:asymptotic}
	
In this section we study the convergence behavior of the method of Theorem \ref{thm:main} by relaxing conditions (i) and (vi). On the other hand, we require summable overrelaxation parameters in condition (ii). By doing so, we may employ a quasi-Fej\'erian analysis to ensure the asymptotic convergence, either weak or strong.

\begin{theorem}\label{thm:main2}
Let $\{x_k\}_{k=0}^\infty$ be a sequence defined  by \eqref{int:xk}--\eqref{int:betak}. Assume that
\begin{enumerate}[(i)]
  \item $C \cap Q \neq \emptyset$,
  \item $\sum_{k=0}^\infty r_k<\infty$,
\end{enumerate}
and that conditions (iii)--(v) of Theorem \ref{thm:main} are satisfied. Moreover, assume that
\begin{enumerate}
  \item[(vi)] $T_i$ are weakly regular, $i\in I$.
\end{enumerate}
If the sequence $\{x_k\}_{k=0}^\infty$ is bounded (see Example \ref{ex:phi2}), then it converges weakly to some $x_\infty \in C\cap Q$. Furthermore, if
\begin{enumerate}
  \item[(vi')] $T_i$ are boundedly regular, $i\in I$, and $\{Q\}\cup\{C_i\:\colon\:i\in I\}$ is boundedly regular,
\end{enumerate}
then the convergence is in norm.
\end{theorem}

\begin{proof}
Let $S = B(z, R)$ be a ball in $\mathcal H$ containing $\{x_k\}_{k=0}^\infty$, centered at some $z \in C$. By assumption (iv) there are $\delta, \Delta > 0$ such that $\delta \leq \varphi_i(x_k) \leq \Delta$ for all $k=0,1,2,\ldots$. We divide the proof into four steps.

\underline{\emph{Step 1.}} We first show that $\{x_k\}_{k=0}^\infty$ is QF1 with respect to $C \cap Q$. Indeed, define
\begin{equation}
  W_k := \sum_{i\in I_k}\lambda_{i,k} T_i \quad \text{and} \quad
	e_k := \sum_{i\in I_k}\lambda_{i,k}\frac{r_k}{\varphi_i(x_k)} \frac{T_i(x_k)-x_k}{\|T_i(x_k)-x_k\|},
\end{equation}
and note that $x_{k+1} = P_Q(x_k + \alpha_k(W_k(x_k) - x_k + e_k))$ for all $k=0,1,2,\ldots$. Moreover, it is not difficult to see that $\sum_{k=0}^\infty \|e_k\| < \infty$ as $\|e_k\| \leq \frac{r_k}{\delta}$. Consequently, in view of Lemma \ref{lem:QFM}, the sequence $\{x_k\}_{k=0}^\infty$ is QF1 with $\varepsilon_k = \alpha_k \|e_k\|$, as asserted.

\underline{\emph{Step 2.}} Note that, by Lemma \ref{lem:QFM}, we obtain $\|W_k(x_k) - x_k\| \to 0$ and $\|x_{k+1} - x_k\| \to 0$ as $k \to \infty$. We also have
\begin{equation}\label{pr:main2:Ti}
  \max_{i\in I_k} \|T_i(x_k) - x_k\| \to 0
\end{equation}
as $k \to \infty$. To see this, it suffices to apply \cite[Theorem 8]{KRZ17} in view of which
\begin{equation}
  \|W_k(x_k)-x_k\|
  \geq \frac{\lambda_{i,k}}{2R} \frac{2-\alpha_k}{\alpha_k} \sum_{i\in I_k}\|T_i(x_x)-x_k\|^2
  \geq \frac{\varepsilon^2}{4R} \sum_{i\in I_k}\|T_i(x_x)-x_k\|^2.
\end{equation}

\underline{\emph{Step 3.}} We are now ready to show that assumption (iii) leads to weak convergence. Indeed, let $x_\infty$ be a weak cluster point of $\{x_k\}_{k=0}^\infty$ and let $x_{n_k}\rightharpoonup x_\infty$. Let $i\in I$ be fixed and for each $k=0,1,2,\ldots$, let $m_k\in\{n_k,\ldots,n_k+s-1\}$ be the smallest natural number such that $i\in I_{m_k}$. The existence of such a sequence is guaranteed by assumption (vi). Obviously, by \emph{Step 2}, we obtain $x_{m_k}\rightharpoonup x_\infty$. Furthermore, using \eqref{pr:main2:Ti}, we deduce that
\begin{equation}
	\|T_i( x_{m_k})-x_{m_k}\| \to 0
\end{equation}
as $k \to \infty$. Hence the assumed weak regularity of $T_i$ yields that $x_\infty \in C_i$. The arbitrariness of $i$ and the fact that $x_{k} \in Q$ imply that $x_\infty \in C \cap Q$. Thus we have shown that each weak cluster point of $\{x_k\}_{k=0}^\infty$ lies in $C\cap Q$ which, by Theorem \ref{thm:QFM}(ii), means that $ x_k \to x_\infty\in C\cap Q$.

\underline{\emph{Step 4.}} Finally, we show that assumption (vi') leads to norm convergence. Indeed, let $i\in I$ and this time, for every $k=0,1,2,\ldots$, let $m_k\in\{k,\ldots,k+s-1\}$ be the smallest natural number such that $i\in I_{m_k}$. The assumed bounded regularity of $T_i$, when combined with \eqref{pr:main2:Ti}, leads to
\begin{equation}\label{oddevengoestozero}
	\max_{i\in I_k} d(x_k,C_i) \to 0
\end{equation}
as $k \to \infty$. Moreover, by the triangle inequality and by the definition of the metric projection, we get
\begin{equation}
  d(x_k, C_i)
  \leq \|x_{k}-P_{C_i}(x_{m_k})\|
	\leq \|x_{k}-x_{m_k}\|+\|x_{m_k}-P_{C_i}(x_{m_k})\| \to 0
\end{equation}
as $k \to \infty$. The arbitrariness of $i$ implies that $\max_{i\in I}d(x_k, C_i) \to 0$ as $k \to \infty$. Furthermore, by definition, $x_k \in Q$. This, when combined with the bounded regularity of the family $\{Q\} \cup \{C_i \colon i \in I\}$ shows that $d(x_k, C \cap Q) \to 0$. By Theorem \ref{thm:QFM} (iii), we conclude that $x_k \to x_\infty \in C \cap Q$, which completes the proof.
\end{proof}

\begin{example}[Bounded Trajectory] \label{ex:phi2}
  Note here that for a constant $\varphi_i(x) := 1$, we have $\|e_k\| \leq r_k$. In particular, assumption (ii) ($\sum_{k=0}^\infty r_k < \infty$) enures that the trajectory $\{x_k\}_{k=0}^\infty$ is QF1 and hence, it must be bounded. This corresponds to Example \ref{ex:phi} (a).
\end{example}

 We finish this section by formulating  a corollary in which the assumptions of Theorems \ref{thm:main} and \ref{thm:main2} overlap.

\begin{corollary} \label{cor:proj}
  Let $\{x_k\}_{k=0}^\infty$ be a sequence defined  by \eqref{int:xk}--\eqref{int:betak} with
  \begin{equation}\label{}
    T_i := P_{C_i}, \quad \varphi_i := 1\quad \text{and} \quad r_k := \frac{1}{k^\alpha},\quad \text{where}\quad \alpha > 1.
  \end{equation}
  Assume that $C \cap Q \neq \emptyset$, $\alpha_k \in [\varepsilon, 2-\varepsilon]$, $\lambda_{i,k} \in [\varepsilon, 1]$ and $\{I_k\}_{k=0}^\infty$ is $s$-intermittent. Then
  \begin{enumerate}[(i)]
    \item $\{x_k\}_{k=0}^\infty$ converges weakly to some $x_\infty \in C \cap Q$.
    \item If $\{C_i \colon i \in I\}$ is boundedly regular, then the convergence is in norm.
    \item If $\interior(C) \cap Q \neq \emptyset$, then the convergence is finite.
  \end{enumerate}
\end{corollary}
	
\bigskip\textbf{Acknowledgements.} We are very grateful to an anonymous
referee for pertinent comments and helpful suggestions.

\bigskip\textbf{Funding.}  This research was supported by the Israel Science Foundation (Grants No. 389/12 and 820/17), the Fund for the Promotion of Research at the Technion and by the Technion General Research Fund.
	
\small
\addcontentsline{toc}{section}{References}

\end{document}